\documentclass[pdftex]{amsart}

\numberwithin{equation}{section}

\usepackage[dvipdfmx]{graphicx}
\usepackage{cleveref}

\newtheorem{proposition}{Proposition}[section]
\newtheorem{theorem}{Theorem}[section]
\newtheorem{corollary}{Corollary}[section]
\newtheorem{remark}{Remark}[section]
\newtheorem{convention}{Convention}
\newtheorem{notation}{Notation}[section]

\newcommand{\pencil}{P}
\newcommand{\bpencil}{BP}

\begin{document}

\title{On the Teichm\"uller space  of acute triangles}
\author{Hideki Miyachi, Ken'ichi Ohshika, and Athanase Papadopoulos}
\address{Hideki Miyachi,
School of Mathematics and Physics,
College of Science and Engineering,
Kanazawa University,
Kakuma-machi, Kanazawa,
Ishikawa, 920-1192, Japan}
\email{miyachi@se.kanazawa-u.ac.jp}  
\address{Ken'ichi Ohshika,
Department of Mathematics,
Gakushuin University,
Mejiro, Toshima-ku, Tokyo, Japan}
  \email{ohshika@math.gakushuin.ac.jp}
\address{Athanase Papadopoulos, 
Institut de Recherche Math\'ematique Avanc\'ee (Universit\'e de Strasbourg et CNRS), 
7 rue Ren\'e Descartes
67084 Strasbourg Cedex France}
 \email{papadop@math.unistra.fr}
\date{\today}
\date{}
\maketitle

\begin{abstract}
We continue the study of the analogue of Thurston's metric on the Teichm\"{u}ller space of Euclidean triangle which was started by Saglam-Papadopoulos in \cite{SP}.
By direct calculation, we give explicit expressions of the distance function and the Finsler structure of the metric restricted to the subspace of acute triangles.
We deduce from the form of the Finsler unit sphere a result on the infinitesimal rigidity of the metric.
We give a description of the maximal stretching loci for a family of extreme Lipschitz maps.


\medskip

\end{abstract}

\section{Introduction}

In his paper \cite{thurston}, William Thurston introduced an asymmetric metric on the Teichm\"uller space of complete  hyperbolic metrics of finite type on a given surface, where the distance between two points is defined as the logarithm of the smallest dilatation constant of a Lipschitz homeomorphism between two marked surfaces equipped with hyperbolic structures representing these points. Thurston discovered several properties of this metric, called now \emph{Thurston's metric}. In particular, he gave a description of a large class of geodesics, called stretch lines,  he showed that any two points are connected by a geodesic, he proved that  this metric is Finsler, and he gave a description of the unit sphere of the conorm of this Finsler structure at each cotangent space of Teichm\"uller space, as an embedding of the space of projective measured laminations on the surface.

In the last couple of decades, Thurston's metric has been adapted to many different settings. We list here some of these settings, to give an indication of some of the  developments: Teichmüller spaces of surfaces with boundary, see e.g. the recent papers  \cite{HP} and  \cite{AD};  Teichmüller space of the torus, see \cite{BPT} and \cite{OMP};  Randers--Teichm\"uller metrics \cite{Randers};
 singular flat metrics underlying a fixed quadrangulation \cite{SPQ}; higher dimensional hyperbolic manifolds \cite{GK};  higher Teichm\"uller theory, see \cite{Tholozan}, and there are others. This list is necessarily very partial, and the work and the literature on this topic is growing at a fast rate.
 See also the list of problems  in \cite{Su}, and \cite{PS} and \cite{Xu} for two recent surveys on Thurston's metric and its developments. 

Recently, \.{I}smail  Sa\u{g}lam and the third author of the present paper developed the theory of Thurston's Lipschitz metric on the Teichmüller space of Euclidean triangles, that is, the moduli space of marked Euclidean triangles (see  \cite{SP}). They proved in particular that restricted to the space of acute triangles, the Lipschitz distance between two marked triangles is equivalent to another distance, defined as the logarithm of the maximum of the lengths of the three edges and the three altitudes of the two triangles.  This new formula   is an analogue of Thurston's alternative description of his distance in terms of  ratios of lengths of simple closed geodesics (see  \cite[p. 4]{thurston}). In \cite{SP}, the authors  also proved that
	for $A>0$, the metric induced on the space of acute triangles having fixed area $A$ is Finsler. They gave a
	 characterisation of geodesics  in $\frak{AT}_A$:  a path is geodesic if and only if the angle at each labelled vertex of triangles in this family varies monotonically. Finally, they proved that the
	 isometry group of $\frak{AT}_A$ is isomorphic to $S_3$, the symmetric group on $\{1,2,3\}$.

The aim of the present paper is to continue studying Thurston's metric on the Teichm\"{u}ller space of Euclidean triangles inaugurated in  \cite{SP}.
Following \cite{SP}, for $A>0$, we let $\mathfrak{T}_A$ and $\mathfrak{AT}_A$ denote respectively the Teichm\"uller spaces of Euclidean triangles and acute triangles of fixed area $A$ (see \S\ref{sec:Teichmuller_space_triangles}).
 Unlike the metric described by Thurston on the Teichm\"uller space of hyperbolic surfaces, the Lipschitz  metric on $\mathfrak{AT}_A$ is symmetric as shown in \cite{SP}. In the present paper, developing the theory of  Thurston's metric on  the Teichm\"uller spaces of triangles further, we obtain the following:
\begin{itemize}
\item
Explicit formulae for the Lipschitz metric and its associated Finsler infinitesimal structure on $\mathfrak{AT}_A$,
after identifying the space $\mathfrak{T}_A$ with the upper-half plane $\mathbb{H}$ in the complex plane $\mathbb{C}$. 

\item A description of the Finsler ball at a point in the tangent space of $\mathfrak{AT}_A$ using the new parameters, and an infinitesimal rigidity result, saying that the infinitesimal unit ball at a give point determines the  triangle as a point in Teichm\"uller space.
\item
A characterisation of the stretch locus for an extremal Lipschitz map between acute triangles, and an analogue of Teichm\"uller's result (respectively Thurston's result) on the existence of foliations (respectively laminations) on which  the Lipschitz constant for the extremal Lipschitz map is optimal. 

\item The perhaps surprising result that, contrary to the usual case of Thurston's metric, there are instances where there are no extremal Lipschitz homeomorphisms between two triangles.
\end{itemize} 

The present setting of the Teichm\"uller space of the triangle is a simple case where   Thurston's ideas on his Lipschitz metric have simple analogues and can be described in an elementary way, without heavy use of Thurston's theory of surfaces.


\section{Teichm\"uller space of triangles}
\label{sec:Teichmuller_space_triangles}

\subsection{labelled triangles}
 We identify the Euclidean plane with the complex plane $\mathbb{C}$ in the canonical way. Consider a triangle in the Euclidean plane. We label its vertices by an ordered triple so that this labelling induces a counter-clockwise orientation on the boundary of the triangle. We call such a triangle \emph{labelled}. Consider a $2$-simplex $\sigma_0$ with vertices labelled as $a$, $b$, $c$, fixed once and for all. The labelled triangle is thought of as a pair $(T,f)$ of a triangle $T$ equipped with a vertex-preserving homeomorphism $f\colon \sigma_0\to T$, called the \emph{marking} of the triangle. In this setting, for $v\in \{a,b,c\}$, the vertex of $T$ corresponding to the vertex labelled as $v$ of $\sigma_0$ via the marking is called the \emph{$v$-vertex} of $T$.

Henceforth, for a labelled triangle $T$ and for $v\in \{a,b,c\}$, we  use the notation $z_v(T)\in \mathbb{C}$ for the $v$-vertex of $T$. Let $\theta_v(T)$ be the angle of $T$ at the $v$-vertex. We denote by $e_v(T)$ the edge of the triangle which is opposite to the $v$-vertex, and by $|e_v(T)|$ its Euclidean length.

\subsection{The space of triangles}
\label{subsec:the_space_of_triangle}
We say that two labelled triangles $T$ and $T'$ are \emph{(Teichm\"uller)-equivalent} if there is a label-preserving isometry (in terms of the Euclidean metric) from $T$ onto $T'$. 

Let $\mathfrak{T}_A$ be the set of Teichm\"uler equivalence classes of triangles of area $A$.
Let $\mathfrak{AT}_A$ be the subset of $\mathfrak{T}_A$ consisting of acute triangles. For $v\in \{a,b,c\}$, we denote by $\mathfrak{OT}_v$, or, more simply,  $\mathfrak{O}_v$, the subset of $\mathfrak{T}_A$ consisting of obtuse triangles with obtuse angle at the $v$-vertex.

\begin{convention}
From now on, to simplify notation, we omit the subscript $A$ (which concerns the area of triangles) from the symbols in the case of unit area triangles (i.e. $A=1$). For instance, we set $\mathfrak{T}=\mathfrak{T}_1$. We also simply write $T$ to denote the Teichm\"uller equivalence class of a triangle $T$. 
\end{convention}

\subsection{Realisation of the space of triangles} In \S 6 of the paper \cite{SP}, the authors suggested an identification between the space  $\frak{AT}_{\frac{1}{2}}$ of acute triangles of area $\frac{1}{2}$ and an ideal triangle in the hyperbolic plane. We shall use this identification.

The space $\mathfrak{T}_A$ is naturally identified with the upper-half-plane $\mathbb{H}\subset \mathbb{C}$, by taking $z\in \mathbb{H}$ to the equivalence class of a labelled triangle $T_z$ with 
\begin{equation}
\label{eq:normalised_parameter_of_triangle}
z_a(T_z)=(2A/y)^{1/2}z, z_b(T_z)=0, z_c(T)=(2A/y)^{1/2}
\end{equation}
with area $A$ in $\mathbb{C}$, where $z=x+iy$. 
We denote by $z$ the inverse of this map, which takes each point in $\mathfrak{T}_A$ to a number in $\mathbb{H}$.
For $T\in \mathfrak{T}_A$ with $z=z(T)$, we call $T_z$ the \emph{normalised position} of $T$. Under this identification, the space $\mathfrak{AT}_A$ of acute triangles is realised as the interior of the ideal (hyperbolic) triangle $\mathfrak{A}$ with vertices $0$, $1$, $\infty$. 
The boundary $\partial \mathfrak{A}$ consists of right triangles. 
For instance, the hyperbolic geodesic connecting $0$ and $1$ parametrises the right triangles of area $A$ whose right angles are at the vertex labelled $a$. Let $\mathfrak{r}_a$ (resp. $\mathfrak{r}_b$, $\mathfrak{r}_c$) denote the complete hyperbolic geodesic connecting $0$ and $1$ (resp. $\infty$ and $0$, $\infty$ and $1$).
The complement of the closure of the ideal triangle with vertices $0, 1$ and $\infty$ consists of obtuse triangles. For a vertex $v\in \{a,b,c\}$, we denote by $\mathfrak{O}_v$ the domain enclosed by $\mathfrak{r}_x$ and $\partial \mathbb{H}$ (see Figure \ref{fig:H_teichmuller}).
\begin{figure}
\includegraphics[height = 5cm]{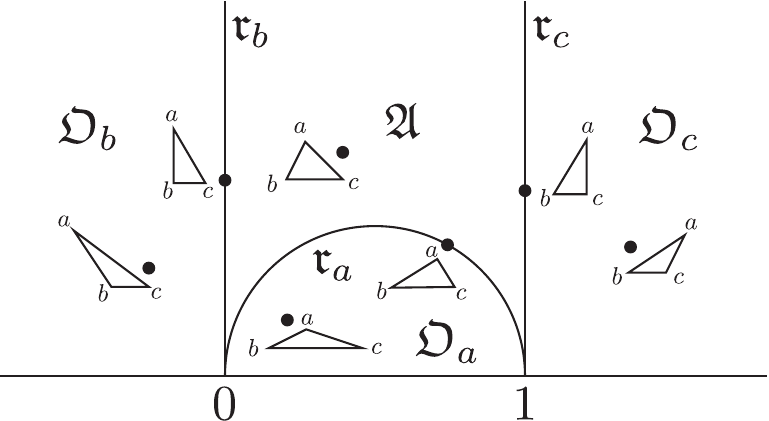}
\caption{The space of triangles $\mathfrak{T}_A$. The hyperbolic ideal triangle $\mathfrak{A}$ with vertice $\{0,1,\infty\}$ is identified with the space $\mathfrak{AT}_A$ of acute triangles. For $v\in \{a,b,c\}$, $\mathfrak{r}_v$ consists of right triangles whose angle is $\pi/2$ at the $v$-vertex, and $\mathfrak{O}_v$ consists of obtuse triangles which have an obtuse angle at the vertex labelled $v$.}
\label{fig:H_teichmuller}
\end{figure}

For $v\in \{a,b,c\}$, we set $\mathfrak{T}_A(v)=\mathfrak{AT}_A\cup \mathfrak{r}_v\cup \mathfrak{OT}_v$.

\subsection{Lipschitz metric}
Consider two labelled triangles $T$ and $T'$ on the plane, and let $f\colon T\to T'$ be a label-preserving continuous map. The Lipschitz constant of $f$ is defined as
$$
L(f)=\sup_{x,y\in T,x\ne y}\dfrac{d_{euc}(f(x),f(y))}{d_{euc}(x,y)}
$$
where $d_{euc}$ denotes the Euclidean metric of the plane. A map $f$ is said to be \emph{Lipschitz} if $L(f)<\infty$. A  label-preserving map $f\colon T\to T'$ is said to be \emph{edge-preserving} if for $v\in \{a,b,c\}$, $f(e_v(T))\subset e_v(T')$. Note that a label-preserving homeomorphism is always edge-preserving.

We define $L(T,T')$ by 
$$
L(T,T'):=\inf\{L(f)\mid \mbox{$f$ is a label-preserving homeomorphism between $T$ and $T'$}\},
$$
and set
$$
d_L(T,T')=\log L(T,T').
$$
A label-preserving and edge-preserving Lipschitz map $g\colon T\to T'$ is said to be \emph{extremal} if $L(T,T')=L(g)$. Notice that an extremal Lipschitz map is not necessarily a homeomorphism.

One can easily see that $d_L$ is a distance function on $\mathfrak{T}_A$, cf. \cite{SP}. We call $d_L$ the \emph{Lipschitz metric} on $\mathfrak{T}_A$. We give a topology the set $\mathfrak{T}_A$ (and hence $\mathfrak{AT}_A$) induced from the Lipschitz metric.  

\begin{convention}
A similarity of the complex plane with factor $\sqrt{A}$ gives a natural isometric identification (with respect to the Lipschitz metric) between $\mathfrak{T}$ and $\mathfrak{T}_A$. For this reason, henceforth, we shall only work with the space of of triangles of unit area.
\end{convention}

\subsection{Parametrisation}
For a labelled triangle $T$ of unit area, the corresponding parameter $z=z(T)$ is expressed by
\begin{equation}
\label{eq:parameter_z}
z(T)=\dfrac{|e_c(T)|}{|e_a(T)|}e^{i\theta_b(T)},\quad
\cos\theta_b(T)=\dfrac{|e_a(T)|^2+|e_c(T)|^2-|e_b(T)|^2}{2|e_a(T)||e_c(T)|}.
\end{equation}
Notice that $T=(a,b,c)$ is  similar to the triangle with (ordered) vertices $(z(T),0,1)$ by a label-preserving similarity map. 
A label-preserving affine map from $T_1=T_{z_1}$ to $T_2=T_{z_2}$ is given by
$$
A_{T_1;T_2}(\zeta)=\sqrt{\dfrac{{\rm Im}(z_1)}{{\rm Im}(z_2)}}
\left(\dfrac{z_2-\overline{z_1}}{z_1-\overline{z_1}}\zeta+\dfrac{z_1-z_2}{z_1-\overline{z_1}}\overline{\zeta}\right).
$$
Its Lipschitz constant, which we denote by $L_{Af}(T_1,T_2)$, is expressed as
\begin{equation}
\label{eq:Lipschitz_constant_affine_map}
L_{Af}(T_1,T_2)=\sqrt{\dfrac{{\rm Im}(z_1)}{{\rm Im}(z_2)}}
\left|
\dfrac{z_2-\overline{z_1}}{z_1-\overline{z_1}}
\right|
+
\left|
\dfrac{z_1-z_2}{z_1-\overline{z_1}}
\right|
=\dfrac{|z_2-\overline{z_1}|+|z_2-z_1|}{2{\rm Im}(z_1)^{1/2}{\rm Im}(z_2)^{1/2}}
\end{equation}
(cf. \cite[\S2]{Saglam1}).
 We note the following.

\begin{proposition}
The parametrisation $z \colon \mathfrak{T}\to \mathbb{H}$ is a homeomorphism.
\end{proposition}

\begin{proof}
We only need to show the bi-continuity. 
Let $T_1$ and $T_2$ be points in $\mathfrak{T}$. By the definition of the Lipschitz metric, we have 
$$
L(T_1,T_2)\ge \left|\max\left\{
\dfrac{|e_a(T_1)|}{|e_{a}(T_2)|}, \dfrac{|e_b(T_1)|}{|e_{b}(T_2)|},\dfrac{|e_c(T_1)|}{|e_{c}(T_2)|} \right\}\right|.
$$
Hence, when a sequence  $\{T_n\} \subset \mathfrak{T}$ converges to $T$ in the Lipschitz metric $L$, the lengths of the edges also converge. By \eqref{eq:parameter_z},  this implies that $\{z(T_n)\}$ converges to $z(T)$ in $\mathbb{H}$.
The continuity of the inverse map $z^{-1}$ follows from \eqref{eq:Lipschitz_constant_affine_map} and the relation $\log L(T_1,T_2)\le \log L_{Af}(T_1,T_2)$ for any pair of triangles $T_1$ and $T_2$.
\end{proof}

\subsection{Pencils and Backward pencils}
For $\xi,\eta\in \partial \mathbb{H}$ with $\xi\ne \eta$, we denote by $\mathfrak{r}= \mathfrak{r}(\xi,\eta)$ the complete hyperbolic geodesic in $\mathbb{H}$ joining them. For $d>0$, we denote by $N(\mathfrak{r},d)$ the $d$-neighborhood of the geodesic $\mathfrak{r}$. 
Let $T$ be a labelled triangle.
Suppose that  $z(T)$ lies in $\mathfrak{T}(v)$ for $v\in \{a,b,c\}$. For the vertex $v$ of $T$, we define the \emph{$v$-pencil} $P(T;v)$ by
$$
\pencil(T;v):=\overline{N(\mathfrak{r}_{v'},d_{v'})\cap N(\mathfrak{r}_{v''},d_{v''})
\cap (\mathfrak{A}\cup \mathfrak{O}_v)}
$$
with $\{v,v',v''\}=\{a,b,c\}$,
where
$
\displaystyle d_w=-\log\tan(\theta_w(T)/2)
$ for $w \in \{v', v''\}$, (cf. Figure \ref{fig:pencil}).

Given  an acute triangle $T$, we define the \emph{$v$-backward pencil} $\bpencil(T;v)$ as
$$
\bpencil(T;v)=\overline{(\mathbb{C}-N(\mathfrak{r}_{v'},d_{v'}))\cap (\mathbb{C}-N(\mathfrak{r}_{v''},d_{v''}))
\cap \mathfrak{A}}
$$
(see Figure \ref{fig:bpencil}).
\begin{figure}
\includegraphics[height = 4cm]{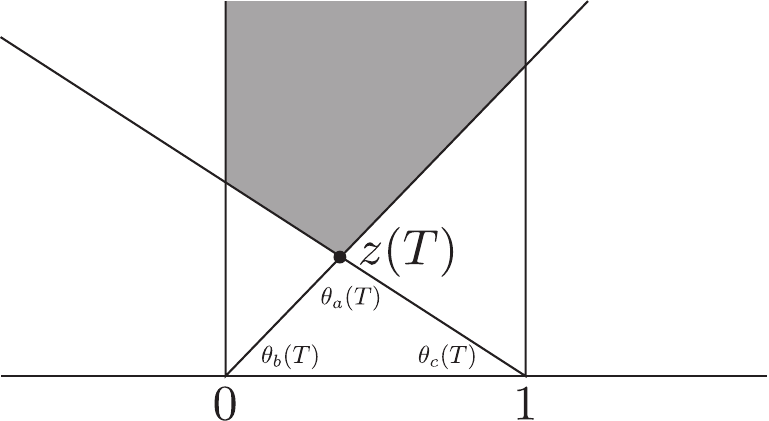}
\caption{The shaded region represents tha $a$-pencil $\pencil(T;a)$ for $z(T)\in \overline{\mathfrak{A}}\cup \mathfrak{O}_a$.}
\label{fig:pencil}
\end{figure}
\begin{figure}
\includegraphics[height = 4cm]{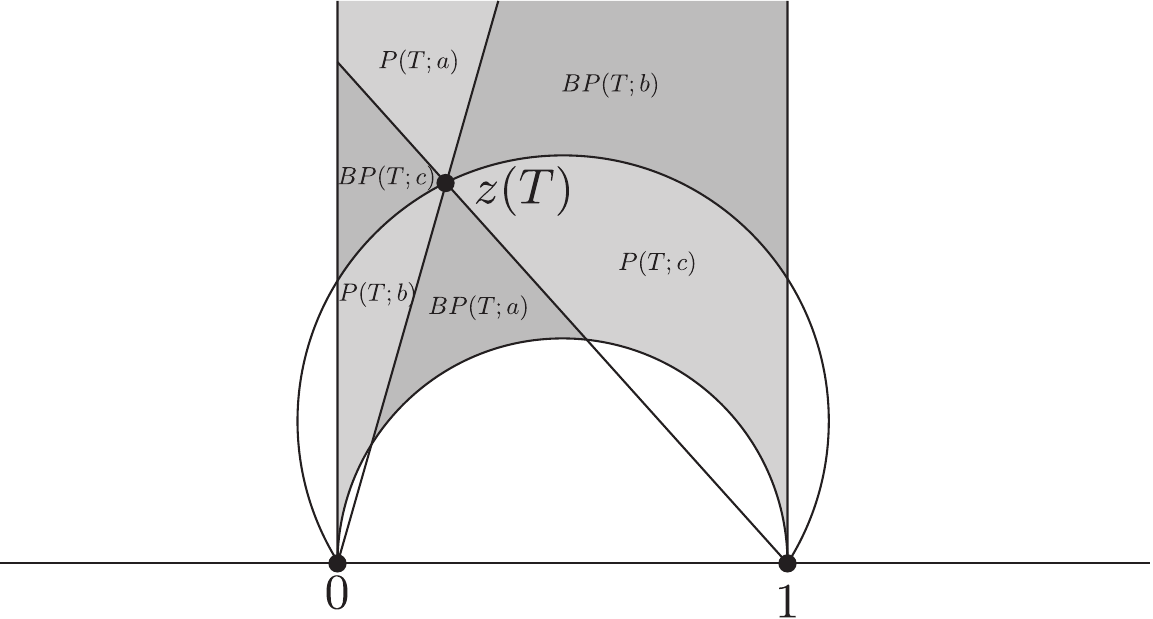}
\caption{$\pencil(T;v)$ and $\bpencil(T;v)$ for $z(T)\in \mathfrak{A}$ and $v\in \{a,b,c\}$.}
\label{fig:bpencil}
\end{figure}
After identifying $\mathfrak{T}$ with $\mathbb{H}$, 
$\pencil(T;v)$ and $\bpencil(T;v)$ are also regarded as subsets of $\mathfrak{T}$ for $v\in \{a,b,c\}$. 

\subsection{Symmetries}
\label{subsec:symmetries}
The symmetric group $\mathfrak{S}_3$ of degree $3$ acts on $\mathfrak{T}$ naturally as permutations of the labelled vertices. This symmetric group is thought of as the mapping class group of a triangle (see \S 6 of the paper \cite{SP}). Under the identification between $\mathfrak{T}$ and $\mathbb{H}$, the action of $\mathfrak{S}_3$ on $\mathbb{H}$ is generated by two transformations
\begin{equation}
\label{eq:change_labels}
\omega_{ab}(z)=\dfrac{\overline{z}}{\overline{z}-1},
\quad
\omega_{ac}(z)=\dfrac{1}{\overline{z}},
\end{equation}
where the former permutes the $a$ and $b$-vertices, and the latter permutes the $a$ and $c$-vertices. 
For $z_0\in \mathbb{H}\cong \mathfrak{T}$, the above actions are induced by the (orientation-reversing) congruences
\begin{align*}
R_{ab}\colon & T_{z_0}\ni \zeta \mapsto 
-\dfrac{|z_0-1|}{\overline{z_0}-1}
\left(
\overline{\zeta}-
\sqrt{\dfrac{2}{{\rm Im}(z_0)}}
\right)
+
\dfrac{|z_0-1|}{\sqrt{2{\rm Im}(z_0)}}
 \in T_{\omega_{ab}(z_0)} \\
R_{ac}\colon &T_{z_0}\ni \zeta \mapsto 
\dfrac{|z_0|}{\overline{z_0}}\overline{\zeta} \in T_{\omega_{ac}(z_0)}.
\end{align*}
For the record, we note that the transformation of $\mathfrak{T}$ corresponding to the permutation between the $b$ and $c$-vertices is expressed as
$$
\omega_{bc}(z)=1-\overline{z},
$$
and it is induced by
the congruence
$$
R_{bc}(\zeta)=\sqrt{\dfrac{2}{{\rm Im}(z_0)}}-\overline{\zeta}.
$$

\section{The Lipschitz distance and its Finsler structure}
In this section, we discuss the Lipschitz metric on the Teichm\"uller space of acute triangles.
The goal of this section is to prove the following theorem, which follows from \Cref{prop:extremal_Lip1} and \Cref{prop:Lip_constant_backward_pencils} given later.
 
\begin{theorem}[An expression of the Lipschitz distance]
\label{thm:main}
Let $T$ be a point in $\mathfrak{A}$, and $T$ a point in  $\overline{\mathfrak{A}}$.
Set $z=z(T)\in \mathbb{H}$,  and suppose that $w=z(T')$ lies in $P(T;x)\cup BP(T;x)$ for $x\in \{a,b,c\}$. 
Then, we have
$$
d_L(T,T')=\begin{cases}
\dfrac{1}{2}\left|\log \dfrac{{\rm Im}(w)}{{\rm Im}(z)}\right| & (\mbox{if $x=a$}) \\
\dfrac{1}{2}\left|\log \dfrac{{\rm Im}(w)}{|w|^2}\dfrac{|z|^2}{{\rm Im}(z)}\right|& (\mbox{if $x=b$}) \\
\dfrac{1}{2}\left|\log \dfrac{{\rm Im}(w)}{|w-1|^2}\dfrac{|z-1|^2}{{\rm Im}(z)}\right| & (\mbox{if $x=c$}).
\end{cases}
$$
\end{theorem}
We shall show \Cref{thm:main} by giving extremal maps concretely for each case.

\subsection*{Finsler structure}
In \cite{SP}, Sa\u{g}lam and Papadopoulos gave an expression of the Finsler structure of the deformation space of acute triangles using the side-length parametrisation. In this section, we shall give a new description of the Finsler structure which arises naturally from our setting, using a direct calculation. It is often interesting to have a formula for the norm of a vector and to draw the unit ball in the tangent space, once we know that a geometrically defined distance function arises from a Finsler structure. This will imply in particular an infinitesimal rigidity result which is similar to the one obtained in \cite{Pan,HOP}.

Let $z_0$ be a point in $\mathfrak{A}$.
Let $\theta_1$ and $\theta_2$ be the arguments of $z_0$ and $z_0-1$ respectively, with $-\pi< \theta_1$, $\theta_2\le \pi$. 
We consider six regions in the tangent space $T_{z_0}\mathfrak{T}=\mathbb C$ divided by three lines passing through the origin with angles $\theta_1$, $\theta_2$, and $\theta_1+\theta_2-\pi$. We denote these regions by $S_x(z_0)$, $S^B_{x}(z_0)$ ($x\in \{a,b,c\}$),  as in Figure \ref{fig:tangent_space}.
\begin{figure}
\includegraphics[height = 7cm]{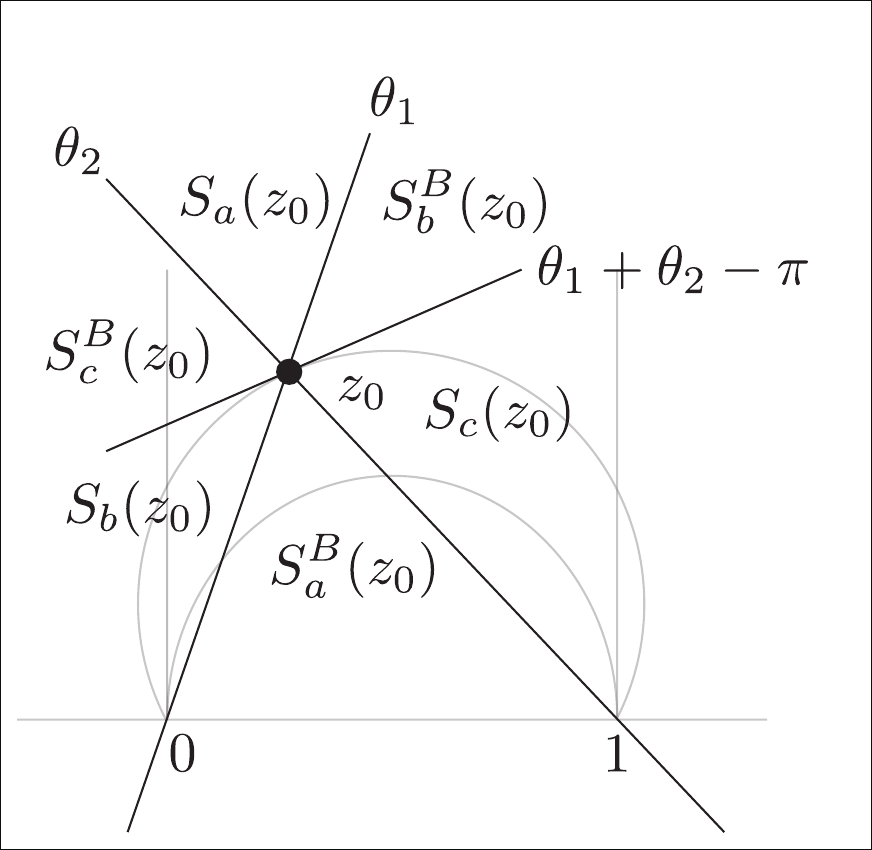}
\label{fig:tangent_space}
\caption{The decomposition of the tangent space at $z_0$ into six sectors}
\end{figure}
We define the norm $F^{\mathfrak{A}}(z_0,v)$ of a vector $v\in T_{z_0}\mathfrak{T} = \mathbb C$ by
$$
F^{\mathfrak{A}}(z_0,v)=
\begin{cases}
\dfrac{1}{2}\dfrac{|{\rm Im}(v)|}{{\rm Im}(z_0)} 
& (v\in S_a(z_0)\cup S^B_a(z_0))
\\[2mm]
\dfrac{1}{2}
\left|\dfrac{{\rm Im}(v)}{{\rm Im}(z_0)}
-2{\rm Re}\left(\dfrac{v}{z_0}\right)
\right|& (v\in S_b(z_0)\cup S^B_b(z_0))
\\[3mm]
\dfrac{1}{2}\left|\dfrac{{\rm Im}(v)}{{\rm Im}(z_0)}
-2{\rm Re}\left(\dfrac{v}{z_0-1}\right)
\right|& (v\in S_c(z_0)\cup S^B_c(z_0)).
\end{cases}
$$
One can check easily that $F^{\mathfrak{A}}$ is continuous on $\mathfrak{A}\times \mathbb{C}$. The Finsler infinitesimal norm is obtained by differentiating the distance in \Cref{thm:main}. Thus, we we have the following.
\begin{theorem}[Finsler structure]
\label{thm:finsler_metric}
The Lipschitz distance $d_L$ on $\mathfrak{A}$ is a Finsler distance with Finsler norm $F^\mathfrak{A}$.
\end{theorem}

\begin{proof}
For any $v\in \mathbb{C}$, we have
\begin{align*}
\dfrac{|{\rm Im}(z_0+tv)|}{{\rm Im}(z_0)} 
&=
1+t\dfrac{{\rm Im}(v)}{{\rm Im}(z_0)} 
\\
\dfrac{{\rm Im}(z_0+tv)}{|z_0+tv|^2}\dfrac{|z_0|^2}{{\rm Im}(z_0)}
&=
\left(
1-2t{\rm Re}\left(\dfrac{v}{z_0}\right)
+o(t)
\right)
\left(
1+t\dfrac{{\rm Im}(v)}{{\rm Im}(z_0)}
\right)
\\
&=1+t\left(\dfrac{{\rm Im}(v)}{{\rm Im}(z_0)}
-2{\rm Re}\left(\dfrac{v}{z_0}\right)\right)+o(t)
\\
\dfrac{{\rm Im}(z_0+tv)}{|z_0+tv-1|^2}\dfrac{|z_0-1|^2}{{\rm Im}(z_0)}
&=
\left(
1-2t{\rm Re}\left(\dfrac{v}{z_0-1}\right)
+o(t)
\right)
\left(
1+t\dfrac{{\rm Im}(v)}{{\rm Im}(z_0)}
\right)
\\
&=1+t\left(
\dfrac{{\rm Im}(v)}{{\rm Im}(z_0)}
-2{\rm Re}\left(\dfrac{v}{z_0-1}\right)
\right)+o(t)
\end{align*}
as $t\to 0$.
Suppose first that $v$ lies in $S_a(z_0)$. Since ${\rm Im}(v)\ge 0$ for $v \in S_a(z_0)$,
 from \Cref{thm:main} and the above calculation, we have
\begin{align*}
d_L(z_0,z_t)
&=\dfrac{1}{2}\dfrac{{\rm Im}(v)}{{\rm Im}(z_0)}+o(t)
=\dfrac{1}{2}\dfrac{|{\rm Im}(v)|}{{\rm Im}(z_0)}+o(t)\quad (t\to 0)
\end{align*}
for a differentiable path in $P(T_{z_0};a)$ which is tangent to $v$ at $t=0$.
We have the same conclusion for a differentiable path in $BP(T_{z_0};a)$ which tangent to $v$ at $t=0$.  

Suppose next that $v$ lies in $S_b(z_0)$. Express $v$ as $v=v_1+iv_2=\xi e^{i\theta}$. 
As depicted in Figure \ref{fig:tangent_space}, we have $\theta_1+\theta_2 - 2\pi \leq \theta \leq \theta_1-\pi$ then, and hence 
$\theta_2-\theta_1-2\pi \le \theta-2\theta_1\le -\pi-\theta_1$.
Then we have 
\begin{align*}
{\rm Im}(v)-
2{\rm Re}(\sin\theta_1e^{-i\theta_1}v)
&=\xi\sin\theta -2\xi \sin\theta_1\cos(\theta-\theta_1)
=\xi \sin(\theta-2\theta_1)>0,
\end{align*}
where the last inequality follows from the fact that both $\sin(\theta_2-\theta_1)$ and $\sin(\pi-\theta_1)$ are positive and  fromthe bounds for $\theta$ given above.
Since $z_0=(e^{i\theta_1}/\sin\theta_1){\rm Im}(z_0)$, we have
\begin{align*}
\dfrac{{\rm Im}(v)}{{\rm Im}(z_0)}-2{\rm Re}\dfrac{v}{z_0}
&=
\dfrac{\xi \sin(\theta-2\theta_1)}
{{\rm Im}(z_0)}>0.
\end{align*}
Hence, from the above calculation, we get, for a differentiable path in $P(T_{z_0}; b)$ tangent to $v$ at $t=0$,
\begin{align*}
d_L(z_0,z_t)
&=t\dfrac{1}{2}\left(\dfrac{{\rm Im}(v)}{{\rm Im}(z_0)}-2{\rm Re}\dfrac{v}{z_0}\right)+o(t) \\
&=t\dfrac{1}{2}\left|\dfrac{{\rm Im}(v)}{{\rm Im}(z_0)}-2{\rm Re}\dfrac{v}{z_0}\right|+o(t).
\end{align*}

The remaining cases  for $BP(T_{z_0};b)$, $P(T_{z_0};c)$, and $PB(T_{z_0};c)$ can be dealt with by a similar argument.
\end{proof}

\subsection*{Some remarks}
Regarding \Cref{thm:main} and \Cref{thm:finsler_metric}, we remark the following.
\begin{enumerate}
\item
  \Cref{thm:finsler_metric} gives another way of seeing the non-uniqueness of geodesics between two points, a fact which was already noticed in \cite{SP}. 
Indeed, we can see that there are uncountably many geodesics between two points in $\mathfrak{A}$.
For instance, take $z_0\in \mathfrak{A}$ and $z_1\in P(T_{z_0};a)$. Then,
a piecewise $C^1$-path $z_t$ ($0\le 1\le 1$) connecting $z_0$ to $z_1$ 
whose imaginary part ${\rm Im}(z_t)$ is increasing is a geodesic between $z_0$ and $z_1$.
Similarly, when $z_1\in BP(T_{z_0};a)$, if the imaginary part is decreasing, $z_t$ is a geodesic between $z_0$ and $z_1$.
\item
For $z_0\in \mathfrak{A}$,
the unit ball associated with the Finsler norm  $F^{\mathfrak{A}}$ at $z_0$ is a hexagon with vertices $2z_0$, $2(z_0-1)$, $2z_0(z_0-1)$, $-2z_0$, $-2(z_0-1)$ and $-2z_0(z_0-1)$. The diagonals divides the unit ball into six triangles, each of which is similar to $T_{z_0}$. See Figure \ref{fig:unit_ball}.
\item
From the second remark, we deduce the following  \emph{infinitesimal rigidity} result: The isometry type  of the unit ball determines uniquely a triangle in the parameter space. 
The same kind of infinitesimal rigidity concerning cotangent spaces of Teichm\"{u}ller spaces of closed surfaces with Thurston's metric was proved in \cite{Pan,HOP}.
With the preceding notation, we can state our infinitesimal rigidity result as follows.
\end{enumerate}

\begin{corollary} Let $z_0$ and $z_1$ be points in $\mathfrak{A}$.
Suppose that there is a linear isometry from $(T_{z_0}\mathfrak{A},F^{\mathfrak{A}}(z_0,\cdot))$ to $(T_{z_1}\mathfrak{A},F^{\mathfrak{A}}(z_1,\cdot))$.
Then the triangles  $T_{z_1}$ and  $T_{z_0}$ are congruent.
\end{corollary}

\begin{figure}
\includegraphics[height = 5cm]{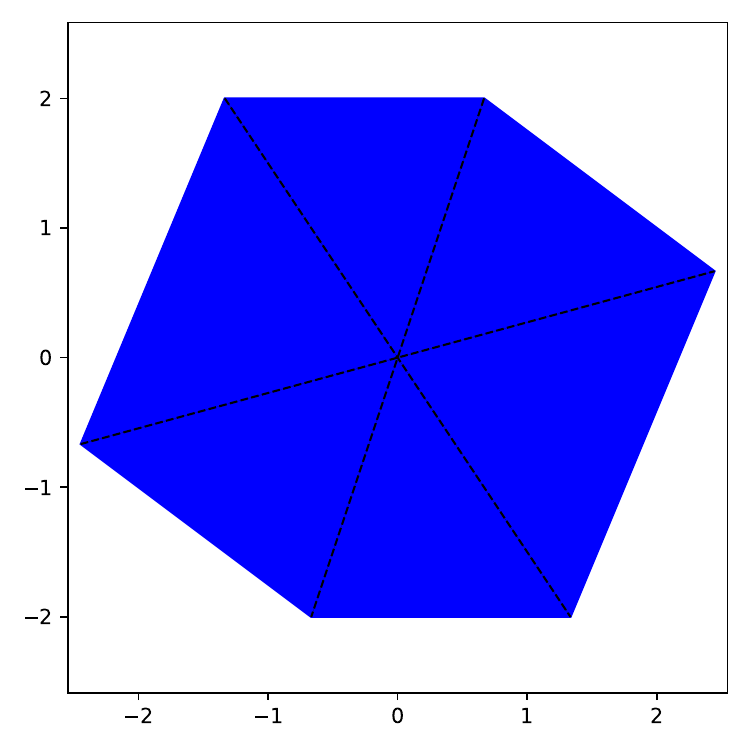}
\label{fig:unit_ball}
\caption{The shape of the Finsler unit ball at $z_0=1/3+i$. The diagonals (dashed lines) divide the hexagon into six triangles and each triangle is similar to $T_{z_0}$.}. 
\end{figure}

\subsection{Stretching loci}
Let us fix an acute triangle $T$. 
With a vertex  $v\in \{a,b,c\}$ and an angle  $\theta \in (\theta_v(T),\pi/2)$, we associate  three kinds of pictures $\mathcal{F}^{v',\theta}_v=\mathcal{F}_{v}^{v',\theta}(T)$ and $\mathcal{G}_v=\mathcal{G}_v(T)$ for the vertex $v$, where $v,v'\in \{a,b,c\}$ are distinct vertices,  as in Figure \ref{fig:foliation3}. The figure only show the case when $v=a$,  but the other cases can be drawn just by changing the symbols of vertices.
\begin{figure}
\includegraphics[height = 4cm]{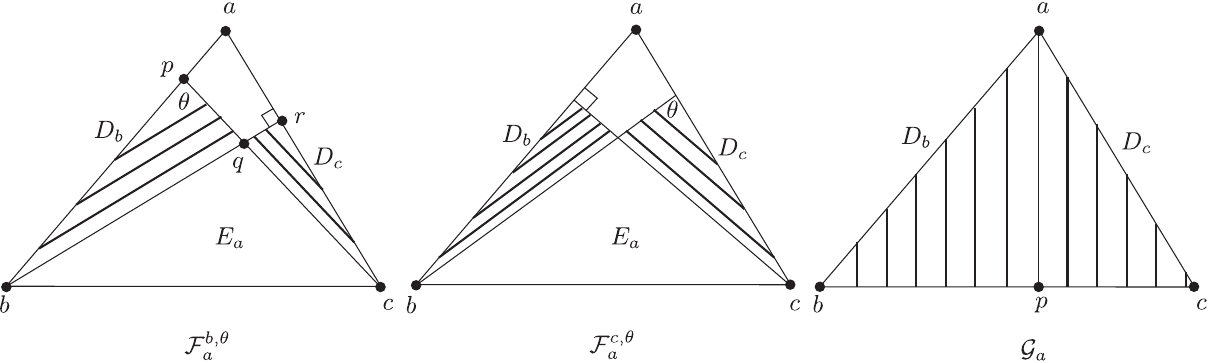}
\caption{The stretching loci $\mathcal{F}_a^{b,\theta}$, $\mathcal{F}_a^{c,\theta}$, $\mathcal{G}_a$}
\label{fig:foliation3}
\end{figure}

In the leftmost triangle of Figure \ref{fig:foliation3}, $\mathcal{F}^{b,\theta}_a$ is drawn as follows. Let $a$, $b$, and $c$ be the $a$-vertex, the $b$-vertex, and the $c$-vertex of $T$, respectively. Draw a perpendicular from $b$ to its opposite side with its foot $r$, and then a segment from $c$ to its opposite side so that the angle $\angle bpc$ is equal to $\theta$, where $p$ is the foot of the segment. 
We denote the intersection between the segments $br$ and $cp$ by $q$.
 Let $E_a$, $D_b$ and $D_c$ be the triangles $\triangle bcq$, $\triangle bpq$, and $\triangle cqr$, respectively. Consider the foliations on $D_b$ (resp. $D_c$) whose leaves are Euclidean segments parallel to the segment $rb$ (resp. the segment $pc$). We call these foliations the \emph{maximally stretched foliations}. The \emph{maximal stretching locus} $\mathcal{F}_a^{b,\theta}$  consists of the foliated triangles $D_b$ and $D_c$, and the triangle $E_a$. We call the triangle $E_a$ the \emph{expanding region}. The maximal stretching locus $\mathcal{F}^{c,\theta}_a$ for the second triangle in Figure \ref{fig:foliation3} is defined in the same way. The names will be justified below (see Remark \ref{remark:2}). 

In the rightmost triangle of Figure \ref{fig:foliation3}, its maximal stretching locus $\mathcal{G}_a$ is drawn as follows. Draw the perpendicular from $a$ to the opposite side, with its foot $p$. Divide the triangle into two right triangles $D_b=\triangle apb$ and $D_c=\triangle apc$. Consider the foliations on $D_b$ and $D_c$ consisting of Euclidean segments parallel to $ap$. The \emph{maximal stretching locus} $\mathcal{G}_a$ consists of foliated triangles $D_b$ and $D_c$. 

\subsection{Extremal Lipschitz maps associated with pencils}
\begin{notation}
\label{z}
For $z\in \mathbb{H}$, we set $\tilde{T}_z$ to be the labelled triangle with vertices 
$$
z_a(\tilde{T}_z)=z, z_b(\tilde{T}_z)=0, z_c(\tilde{T}_z)=1.
$$
\end{notation}

In this section, we discuss an extremal Lipschitz map from a normalised triangle $T$ to a normalised triangle in the pencil $P(T;v)$ for $v\in \{a,b,c\}$. 
We use the notation \ref{z}, and shall discuss only the case when $v=a$. In the case when $v=b$ and $c$, we can define the Lipschitz maps
 in the same way, taking conjugates by congruences  (cf. \S\ref{subsec:symmetries}).

\subsubsection{}
Let $T_0$ and $T_1$ be triangles with $z_0=x_0+iy_0=z(T_0)\in \mathfrak{T}(a)$, and $z_1=u_0+iv_0=z(T_1)\in P(T_0;a)$.

We define
\begin{align*}
f_{z_0,z_1}(\xi+i\eta)&=\begin{cases}
\sqrt{\dfrac{y_0}{v_0}}
\left(
\dfrac{u_0}{x_0}\,\xi+
i\dfrac{v_0}{y_0}\,\eta
\right)
\quad (\xi+i\eta\in \overline{D_b})
\\
\sqrt{\dfrac{2}{v_0}}\left(
u_0+
\sqrt{\dfrac{y_0}{2}}
\left(\dfrac{1-u_0}{1-x_0}\left(
\xi-\sqrt{\dfrac{2}{y_0}}x_0
\right)+i\dfrac{v_0}{y_0}\eta
\right)
\right)
\quad (\xi+i\eta\in \overline{D_c}).
\end{cases}
\end{align*}
One can check that the map $f_{z_0,z_1}$ is a label-preserving Lipschitz map from $T_{z_0}$ to $T_{z_1}$.

\subsubsection{}
For $0\le k_1,k_2\le 1$, we define a \emph{stretch map $f_{\mathcal{G}_v,(k_1,k_2)}$ associated with $\mathcal{G}_a$ and $(k_1,k_2)$} on $T$ as follows.

Let $T\in \mathfrak{T}(a)$ be in normalised position.  We define
\begin{equation}
\label{eq:Stretch_map}
f_{\mathcal{G}_a,(k_1,k_2)}(\xi+i\eta)
=\begin{cases}
\dfrac{k_1\xi+i\eta}{(k_1x_0+k_2(1-x_0))^{1/2}}
& \left(\mbox{if $\xi+i\eta\in \overline{D_b}$}\right) \\[3mm]
\dfrac{k_1\xi_0+k_2(\xi-\xi_0)+i\eta}{(k_1x_0+k_2(1-x_0))^{1/2}}
& \mbox{($\xi+i\eta\in \overline{D_c}$)},
\end{cases}
\end{equation}
where 
$\xi_0=\sqrt{2/y_0}x_0$ (cf. \eqref{eq:normalised_parameter_of_triangle}).
Since
\begin{align*}
f_{\mathcal{G}_a,(k_1,k_2)}(z_a(T))
&=\left(\dfrac{2}{y_0}\right)^{1/2}\dfrac{k_1x_0+iy_0}{(k_1x_0+k_2(1-x_0))^{1/2}} \\
f_{\mathcal{G}_a,(k_1,k_2)}(z_b(T))
&=0\\
f_{\mathcal{G}_a,(k_1,k_2)}(z_c(T))
&=\left(\dfrac{2}{y_0}\right)^{1/2}(k_1x_0+k_2(1-x_0)))^{1/2},
\end{align*}
$f_{\mathcal{G}_a,(k_1,k_2)}$ coincides with $f_{z_0,z_1}$, which maps $T$ to a triangle $T_1$ with
\begin{equation}
\label{eq:stretch_parameter_of_z}
z_1=z(T_1)=\dfrac{k_1x_0+iy_0}{k_1x_0+k_2(1-x_0)}.
\end{equation}
 We  denote such a triangle $T_1$ by ${\rm St}(T,\mathcal{G}_a,(k_1,k_2))$.
  We note that $f_{\mathcal{G}_a,(k_1,k_2)}$ is label-preserving and edge-preserving, and that $f_{\mathcal{G}_a,(1,1)}$ is the identity map.
Note also that $f_{\mathcal{G}_a,(0,k_2)}$ and $f_{\mathcal{G}_a,(k_1,0)}$ are not homeomorphisms. 
In these two cases, the images are right triangles and the maps are not injective.

\begin{remark}
\label{remark:2}
Let $\zeta_1,\zeta_2$ be points on $T$ with ${\rm Re}(\zeta_1)={\rm Re}(\zeta_2)$.
Then, we have
\begin{align*}
d_{euc}(f_{\mathcal{G}_a,(k_1,k_2)}(\zeta_1), f_{\mathcal{G}_a,(k_1,k_2)}(\zeta_2))
&=\dfrac{1}{(k_1x_0+k_2(1-x_0)))^{1/2}}d_{euc}(\zeta_1,\zeta_2)\\
&>d_{euc}(\zeta_1,\zeta_2)
\end{align*}
when $k_1k_2<1$ since $0<x_0$, $1-x_0<1$. Therefore, $f_{\mathcal{G}_a,(k_1,k_2)}$ stretches $T$ along the foliations in $D_b$ and $D_c$. This is the reason why $\mathcal{G}_a$ (and hence $\mathcal{G}_v$ for $v=b$, $c$) is called the maximal stretching locus, and the map $f_{\mathcal{G}_a,(k_1,k_2)}$ (and hence $f_{\mathcal{G}_v,(k_1,k_2)}$ for $v=b$, $c$) is named a stretch map.
\end{remark}

We have the following.

\begin{proposition}[Parameterisations of Pencils]
\label{prop:parameterisation_pencils}
For $v\in \{a,b,c\}$ and $T\in \mathfrak{T}(v)$, 
$z({\rm St}(T,\mathcal{G}_v,(k_1,k_2))))$ is contained in $\pencil(T;v)$ for any $0\le k_1,k_2\le 1$ with $|k_1|+|k_2|>0$.
Conversely, for any $z\in  \pencil(T;v)$, there is a unique pair $(k_1,k_2)$ with $0\le k_1,k_2\le 1$ with $|k_1|+|k_2|>0$ such that $z=z({\rm St}(T,\mathcal{G}_v,(k_1,k_2)))$. 
Therefore, the map
$$
[0,1]\times [0,1]\setminus \{(0,0)\}\ni (k_1,k_2)\mapsto {\rm St}(T,\mathcal{G}_v,(k_1,k_2))\in \pencil(T;x)
$$
is a homeomorphism.
\end{proposition}

\begin{proof}
We shall prove the proposition only in the case when $v=a$.
Set $z(T)=x_0+iy_0$ and $T'=
{\rm St}(T,\mathcal{G}_v,(k_1,k_2))$ as above. From \eqref{eq:stretch_parameter_of_z}, we have $0\le {\rm Re}(z(T'))\le 1$. Since $0\le k_1$, $k_2\le 1$ and $0\le x_0\le 1$, we have
\begin{align*}
\dfrac{{\rm Im}(z(T'))}{{\rm Re}(z(T'))}
&=\dfrac{y_0}{k_1x_0}\ge \dfrac{y_0}{x_0}=\dfrac{{\rm Im}(z(T))}{{\rm Re}(z(T))} \ \hbox{and}
\\
 \dfrac{{\rm Im}(z(T'))}{{\rm Re}(z(T'))-1}
&=-\dfrac{y_0}{k_2(1-x_0)}\le -\dfrac{y_0}{1-x_0}=\dfrac{{\rm Im}(z(T))}{{\rm Re}(z(T))-1},
\end{align*}
 hence $z(T')\in \pencil(T;a)$.

Let $z=x+iy$ be a point in $\pencil(T;a)$. 
By solving the equation $z=z({\rm St}(T,\mathcal{G}_a,(k_1,k_2)))$, we obtain
$$
k_1=\dfrac{xy_0}{x_0y}, \ \ k_2=\dfrac{(1-x)y_0}{(1-x_0)y}.
$$
The condition $z=x+iy\in \pencil(T;a)$ implies that $0\le k_1,k_2\le 1$. Furthermore, $k_1=0$ implies that $z({\rm St}(T,\mathcal{G}_a,(k_1,k_2)))\in \mathfrak{r}_b$, and $k_2=0$ implies that $z({\rm St}(T,\mathcal{G}_a,(k_1,k_2)))\in \mathfrak{r}_c$. Therefore, $k_1$ and $k_2$ cannot vanish simultaneously.
\end{proof}

\begin{proposition}[Stretch maps are extremal]
\label{prop:extremal_Lip1}
For $v\in \{a,b,c\}$ and $0\le k_1$, $k_2\le 1$ with $|k_1|+|k_2|>0$, the stretch map $f_{\mathcal{G}_v,(k_1,k_2)}$ is an extremal Lipschitz map.
\end{proposition}

\begin{proof}
As in the proof of \Cref{prop:parameterisation_pencils}, we only consider the case when $v=a$. 
By definition, $f_{\mathcal{G}_a,(k_1,k_2)}$ is  the composition of a contraction of the horizontal direction with Lipschitz constant $1$ and an expansion with factor $1/(x_0k_1+(1-x_0)k_2))^{1/2}$. Furthermore, the contraction preserves distances in the imaginary direction. 
Hence we have
\begin{equation}
\label{eq:Lip_stretch_map_G}
L(f_{\mathcal{G}_a,(k_1,k_2)})=\dfrac{1}{(x_0k_1+(1-x_0)k_2))^{1/2}}.
\end{equation}
On the other hand, the altitudes of $T$ and $T'={\rm St}(T,\mathcal{G}_a,(k_1,k_2))$ are $(2Ay_0)^{1/2}$ and $(2Ay_0)^{1/2}/(x_0k_1+(1-x_0)k_2))^{1/2}$. Since each altitude is equal to the length of the perpendicular from the $a$-vertex to its opposite side, the Lipschitz constant of any label-preserving and edge-preserving Lipschitz map from $T$ to $T'$ is at least $1/(x_0k_1+(1-x_0)k_2))^{1/2}$.
\end{proof}

\begin{corollary}[Lipschitz distance for pencils]
\label{coro:distance}
For $v\in \{a,b,c\}$ and $0\le k_1$, $k_2\le 1$ with $|k_1|+|k_2|>0$, we have
\begin{equation}
\label{eq:distance_pencil}
L(T,{\rm St}(T,\mathcal{G}_v,(k_1,k_2)))=
\begin{cases}
\dfrac{1}{2}\log \dfrac{1}{x_0k_1+(1-x_0)k_2} & (\mbox{if $v=a$}) \\
\dfrac{1}{2}\log \dfrac{|z_0|^2}{x_0k_1+(|z_0|^2-x_0)k_2} & (\mbox{if $v=b$}) \\
\dfrac{1}{2}\log \dfrac{|z_0-1|^2}{(|z_0|^2-x_0)k_1+(1-x_0)k_2} & (\mbox{if $v=c$}) \\
\end{cases}
\end{equation}
where we set $z(T)=z_0=x_0+iy_0$. In particular, if we set $z_1=z({\rm St}(T,\mathcal{G}_v,(k_1,k_2)))$, we obtain
\begin{equation}
\label{eq:distance_pencil2}
L(T_{z_0},T_{z_1})=
\begin{cases}
\dfrac{1}{2}\log \dfrac{{\rm Im}(z_1)}{{\rm Im}(z_0)} & (\mbox{if $v=a$}) \\
\dfrac{1}{2}\log
\dfrac{1}{2}\left|\log \dfrac{{\rm Im}(z_1)}{|z_1-1|^2}\dfrac{|z_0-1|^2}{{\rm Im}(z_0)}\right| & (\mbox{if $v=b$}) \\
\dfrac{1}{2}\log \dfrac{{\rm Im}(z_1)}{|z_1|^2}\dfrac{|z_0|^2}{{\rm Im}(z_0)} & (\mbox{if $v=c$}) \\
\end{cases}
\end{equation}
\end{corollary}

\begin{proof}
We shall only give a proof in the case when $v=a$ (cf. \eqref{eq:change_labels}). 
If $k_1k_2\ne 0$, the stretch map $f_{\mathcal{G}_a,(k_1,k_2)}$ is a homeomorphism. 
Therefore, \eqref{eq:distance_pencil} and \eqref{eq:distance_pencil2} follow from \eqref{eq:Lip_stretch_map_G}.

Assume now that $k_1=0$. For $\epsilon\ge 0$, we set $T_\epsilon={\rm St}(T,\mathcal{G}_a,(\epsilon,k_2))$. We take a label-preserving affine map $A_\epsilon \colon T_\epsilon\to T_0$ and set $g_\epsilon=A_\epsilon\circ f_{\mathcal{G}_a,(\epsilon,k_2)}$. Then, $g_\epsilon$ is a label-preserving Lipschitz homeomorphism from $T$ to $T_0$. From \eqref{eq:Lipschitz_constant_affine_map} and \Cref{prop:extremal_Lip1}, we have 
$$
L(f_{\mathcal{G}_a,(0,k_2)}) \le
L(T,T_0)\le L(g_\epsilon)\le L(A_\epsilon)L(f_{\mathcal{G}_a,(\epsilon,k_2)})
\to L(f_{\mathcal{G}_a,(0,k_2)})
$$
as $\epsilon\to 0$, which implies \eqref{eq:distance_pencil} and \eqref{eq:distance_pencil2} for the case when $k_1=0$. 
The case when $k_2=0$ can be dealt with in the same way.
\end{proof}

\subsubsection{Extremal Lipschitz maps are not always homeomorphisms}
Let $v$, $v'$ and $v''$ be vertices such that $\{v,v',v''\}=\{a,b,c\}$. 
Let $T$ be a triangle in $\mathfrak{AT}$. For $T'\in \pencil(T;v)\cap (\mathfrak{r}_{v'}\cup \mathfrak{r}_{v''})$, our stretch map from $T$ to $T'$ is not a homeomorphism. 
In fact, we have the following general statement.

\begin{proposition}
\label{prop:no_Lip1}
In the above setting, there is no extremal label-preserving Lipschitz homeomorphism between $T$ and $T'$. 
\end{proposition}

\begin{proof}
Let $g\colon T\to T'$ be a label-preserving Lipschitz homeomorphism. We may assume that $v=a$ and $T'\in \mathfrak{r}_b$. 
As was shown in the proof of \Cref{prop:extremal_Lip1}, the Lipschitz constant of the stretch map is the ratio of the altitudes of $T$ and $T'$. 

Let $p\in e_a(T)$ be the foot of the perpendicular in $T$ from the $a$-vertex $z_a(T)$ to the opposite edge $e_a(T)$. Since $T$ is an acute triangle, $p$ is different from the $b$-vertex $z_b(T)$ of $T$. Since $g$ is a homeomorphism, $g(p)\ne z_b(T')$ but $g(p)\in e_a(T')$. Therefore,
$$
L(T,T')=\dfrac{d_{euc}(z_a(T'),z_b(T'))}{d_{euc}(z_a(T),p)}
<\dfrac{d_{euc}(z_a(T'),g(p))}{d_{euc}(z_a(T),p)}\le L(g).
$$
Hence $g$ is not extremal.
\end{proof}

\subsection{Extremal Lipschitz maps associated with backward pencils}
\subsubsection{Contractions}
Let $z_0\in \mathbb{H}$ be a point in the ideal triangle with vertices $0$, $1$ and $\infty$.
Let $z_1=u_0+iv_0$ be a point in $BP(T;a)$. 
We divide $\tilde{T}_{z_0}$ into two right angled triangles by the perpendicular from the $b$-vertex $0$ of $\tilde{T}_{z_0}$ and  its $ac$-side.
For $x\in \{a,c\}$,
we denote by $D^x_b=D^x_b(\tilde{T}_{z_0})$ the component which contains the $x$-vertex in its boundary. We define the \emph{contraction $C^b_{z_0,z_1}$ on $\tilde{T}_{z_0}$ associated with the $b$-vertex} by
$$
C^b_{z_0,z_1}(\zeta)=
\begin{cases}
\zeta & (\zeta\in D_b^c) \\
\dfrac{1+k}{2}\zeta+
\dfrac{1-k}{2}\dfrac{\zeta_0}{\overline{\zeta_0}}\,\overline{\zeta}
& (\zeta\in D_b^a), 
\end{cases}
$$
where $\zeta_0=i \dfrac{{\rm Im}(z_0)}{1-\overline{z_0}}$ and
$\displaystyle
k = \dfrac{y_0}{|z_0|^2-x_0}\dfrac{{\rm Re}(z_1(\overline{z_0}-1))}{{\rm Im}(z_1(\overline{z_0}-1))}.
$
Notice that $\zeta_0$ is the foot of the perpendicular from the $b$-vertex $0$ of $\tilde{T}_{z_0}$ to the $ac$-side. 
We can check easily that $0 < k\le 1$. Indeed, by assumption, $z_1\in D_b^a$. 
By a calculation, we can see that a Euclidean ray emanating from the $b$-vertex $0$ passing by $z_1$ intersects the $ac$-side of $\tilde{T}_{z_0}$ at
\begin{equation}
\label{eq:contraction_zeta_1}
\zeta_1=\dfrac{y_0}{{\rm Im}(z_1(\overline{z_0}-1))}w_0.
\end{equation}
We can also see that $k=|\zeta_1-\zeta_0|/|z_0-\zeta_0|$, and hence that $0<k\le 1$. 
By definition, $C^b_{z_0,z_1}$ is  a contraction from the triangle with vertices $0$, $z_0$ and $\zeta_0$ to the one with vertices $0$, $\zeta_1$ and $\zeta_0$,
and takes $\tilde{T}_{z_0}$ to $\tilde{T}_{\zeta_1}$.
See the map represented by the right-lower arrow in Figure \ref{fig:Stretch_F}.

We define the \emph{contraction $C^c_{z_0,z_1}$ associated with the $c$-vertex} by
$$
C^c_{z_0,z_1}(\zeta)=
1-\overline{C^b_{1-\overline{z_0},1-\overline{z_1}}(1-\overline{\zeta})},
$$
 obtained by conjugating the contraction $C^b_{\omega_{bc}(z_0), \omega_{bc}(z_0)}$ on the vertical line in $\mathbb{C}$ passing through the midpoint $1/2$ between the $b$-vertex $0$ and the $c$-vertex $1$ of $\tilde{T}_{z_0}$
(see \S\ref{subsec:symmetries}).

We define the \emph{contraction $G^b_{z_0,z_1}$ from $\tilde{T}_{z_0}$ and $\tilde{T}_{z_1}$ associated with the $b$-vertex} by
$$
G^b_{z_0,z_1}(\zeta)=C^c_{\zeta_1,z_1}\circ C^b_{z_0,z_1}(\zeta)
$$
for $\zeta_1$ defined in \eqref{eq:contraction_zeta_1}.

We also define the \emph{contraction $G^c_{z_0,z_1}$ from $\tilde{T}_{z_0}$ and $\tilde{T}_{z_1}$ associated with the $b$-vertex} and also satisfying the properties of \Cref{prop:Lip_contractions} just by interchanging the roles of the $b$ and $c$-vertices. 

The following proposition is an immediate consequence of the definition.

\begin{proposition}[Lipschitz constants of contractions]
\label{prop:Lip_contractions}
Let $z_0$ be a point in  $\mathfrak{A}$ and $z_1$ a point in $BP(T_{z_0};a)$.
Let $x$ be either $b$ or $c$.
\begin{itemize}
\item[(1)]
The Lipschitz constant of the contraction $G^x_{z_0,z_1}$ is 1. 
\item[(2)]
After identifying $T_{z_0}$ with $\tilde{T}_{z_0}$ by a similarity with factor $\sqrt{{\rm Im}(z_0)/2}$, the Lipschitz constant of the contraction  is attained at two points both of which lie in either the expanding region $E_a$ or on any leaf of the foliation $\mathcal{F}^{x,\theta}_a$, where $\theta$ is the angle at $\zeta_1$ of $\tilde{T}_{\zeta_1}$.
\end{itemize}
\end{proposition}

\subsubsection{Extremal maps}
Let $T$ be a (marked) acute triangle and $z_0=x_0+iy_0=z(T)\in \mathfrak{A}$.
Let $z_1=u_0+iv_0\in BP(T;a)$.
We define the piecewise linear maps $g^b_{z_0,z_1}$ and $g^c_{z_0,z_1}$ from $T_{z_0}$ to $T_{z_1}$ by
$$
g^x(\zeta)=\sqrt{\dfrac{2}{v_0}}
G^b_{z_0,z_1}
\left(
\sqrt{
\dfrac{y_0}{2}
}
\zeta
\right)
$$
for $x=b$ or $c$.
\begin{proposition}[Extremal Lipschitz maps for backward pencils]
\label{prop:Lip_constant_backward_pencils}
Let $z_0\in \mathfrak{A}$ and $z_1\in BP(T_{z_0};a)$.
For $x\in \{b,c\}$, we have the following:
\begin{itemize}
\item[(1)]
The contraction $g^x_{z_0,z_1}$ is an extremal Lipschitz map from $T_{z_0}$ to $T_{z_1}$ with Lipschitz constant $\sqrt{{\rm Im}(z_0)/{\rm Im}(z_1)}$.
\item[(2)]
The Lipschitz constant is attained by two points both of which lie either in the expanding region $E_a$ or on each leaf of the foliation $\mathcal{F}^{x,\theta}_a$, where $\theta$ is defined in the same manner as in \Cref{prop:Lip_contractions}.
\end{itemize}
\end{proposition}

\begin{proof}
From \Cref{prop:Lip_contractions}, we need only verify that $g^x_{z_0,z_1}$ is extremal.
Let $g\colon T_{z_0}\to T_{z_1}$ be a label-preserving and edge-preserving Lipschitz map. Since the length of the edge $e_a(T_{z_i})$ is $\sqrt{2/{\rm Im}(z_i)}$ for $i=0,1$, the Lipschitz constant of $L(g)$ satisfies
$$
L(g)\ge \dfrac{\sqrt{2/{\rm Im}(z_1)}}{\sqrt{2/{\rm Im}(z_0)}}
=\sqrt{\dfrac{{\rm Im}(z_1)}{{\rm Im}(z_0)}}=L(g^x_{z_0,z_1}),
$$
which implies what we wanted.
\end{proof}

%
%
\begin{figure}
\includegraphics[height = 8cm]{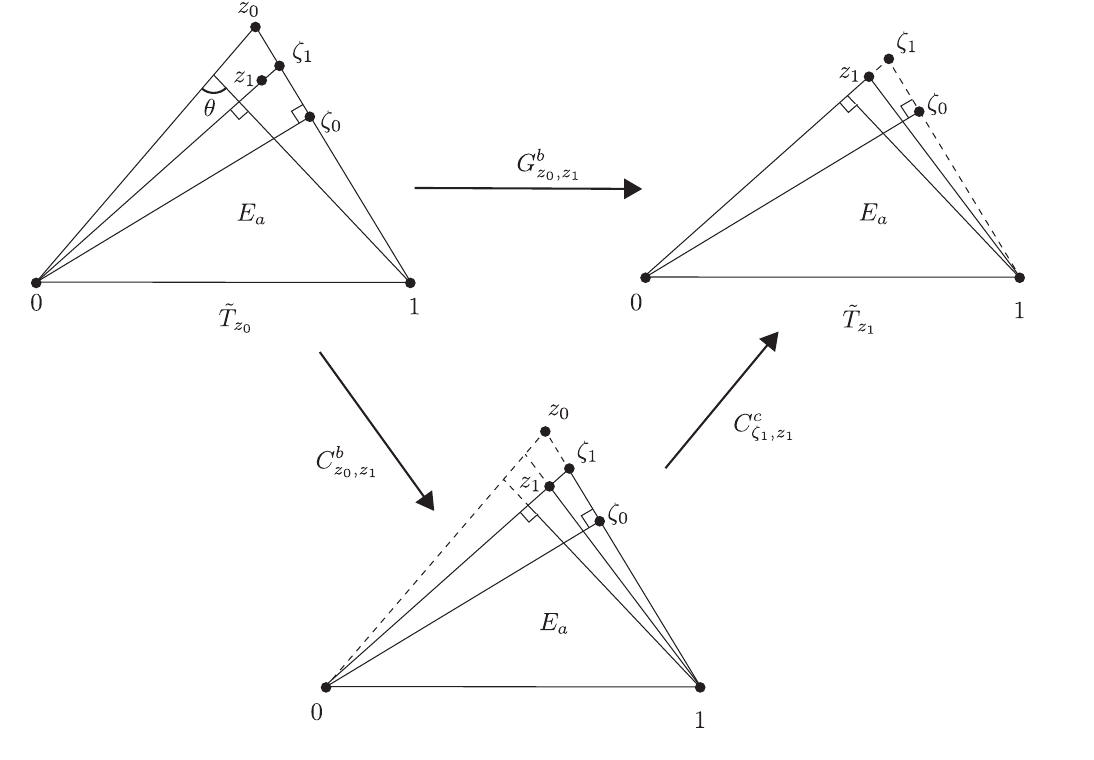}
\caption{Contractions $C^b_{z_0,z_1}$, $C^c_{z_0,z_1}$ and $G^b_{z_0,z_1}$.}
\label{fig:Stretch_F}
\end{figure}

\medskip

\noindent{\bf Acknowledgements} This paper was written during a stay of the three authors at Institut Henri Poincaré (Paris). The authors thank this institute for its support. Ken'ichi Ohshika was partially supported by the Fund for the Promotion of Joint International Research (Japan), 18KK0071. Hideki Miyachi was partially supported by the Fund for the Promotion of Joint International Research (Japan), 20H01800.

\end{document}